\documentclass[a4paper,twoside,10pt]{amsart}
\usepackage{amsmath,amsfonts,amssymb,amsthm}
\usepackage{geometry}
\usepackage{graphicx}
\newtheorem{theorem}{Theorem}[section]

\newtheorem{corollary}{Corollary}[section]
\numberwithin{equation}{section}

\begin{document}

\title[Coefficients of asymptotic expansion]{On the coefficients of\\ the asymptotic expansion of $n!$}
\author{Gerg\H{o} Nemes}	
\address{Lor\'and E\"otv\"os University, H-1117 Budapest, P\'azm\'any P\'eter s\'et\'any 1/C, Hungary}
\email{nemesgery@gmail.com}
\date{}

\begin{abstract}
Applying a theorem of Howard for a formula recently proved by Brassesco and M\'{e}ndez, we derive new simple explicit formulas for the coefficients of the asymptotic expansion to the sequence of factorials. To our knowledge no explicit formula containing only the four basic operations was known until now.
\end{abstract}

\subjclass[2000]{11B65; 11B73; 41A60}

\keywords{asymptotic expansions; factorial; Stirling coefficients; Stirling's formula; Stirling numbers}

\maketitle

\section{Introduction}

It is well known that the factorial of a positive integer $n$ has the asymptotic expansion
\begin{equation}\label{eq1}
n! \sim n^n e^{ - n} \sqrt {2\pi n} \sum\limits_{k \ge 0} {\frac{{a_k }}{{n^k }}} ,
\end{equation}
known as the Stirling's formula (see, e.g., \cite{ref1,ref3,ref4}). The coefficients $a_k$ in this series are usually called the Stirling coefficients \cite{ref1,ref6} and can be computed from the sequence $b_k$ defined by the recurrence relation
\begin{equation}
b_k = \frac{1}{{k + 1}}\left( {b_{k - 1}  - \sum\limits_{j = 2}^{k - 1} {jb_j b_{k - j + 1} } } \right),\;b_0 = b_1 = 1,
\end{equation}
as $a_k  = \left( {2k + 1} \right)!!b_{2k + 1}$ \cite{ref3,ref4}. It was pointed out by Paris and Kaminski \cite{ref6} that ``There is no known closed-form representation for the Stirling coefficients''. However there is a closed-form expression that involves combinatorial quantities due to Comtet \cite{ref5}:
\begin{equation}
a_k  = \sum\limits_{j = 0}^{2k} {\left( { - 1} \right)^j \frac{{d_3 \left( {2k + 2j,j} \right)}}{{2^{k + j} \left( {k + j} \right)!}}},
\end{equation}
where $d_3\left(p,q\right)$ is the number of permutations of $p$ with $q$ permutation cycles all of which are $\geq 3$. Brassesco and M\'{e}ndez proved in a recent paper \cite{ref7} that
\begin{equation}
a_k  = \sum\limits_{j = 0}^{2k} {\left( { - 1} \right)^j \frac{{S_3 \left( {2k + 2j,j} \right)}}{{2^{k + j} \left( {k + j} \right)!}}},
\end{equation}
where $S_3\left(p,q\right)$ denotes the $3$-associated Stirling numbers of the second kind. We show that the Stirling coefficients $a_k$ can be expressed in terms of the conventional Stirling numbers of the second kind. A corollary of this result is that the Stirling coefficients have a representation that involves only the four basic operations, i.e., an explicit, exact expression.

\section{The formulas for coefficients}

One of our main results is the following:
\begin{theorem}\label{thm1} The Stirling coefficients have a representation of the form
\begin{equation}
a_k  = \frac{\left(2k\right)!}{2^k k!}\sum\limits_{i = 0}^{2k} {\binom{k + i - 1/2}{i}\binom{3k + 1/2}{2k - i}2^{i}\sum\limits_{j = 0}^i {\binom{i}{j}\left( { - 1} \right)^j j!} \frac{{S\left( {2k + i + j,j} \right)}}{{\left( {2k + i + j} \right)!}}},
\end{equation}
where $S\left(p,q\right)$ denotes the Stirling numbers of the second kind.
\end{theorem}
From the explicit formula
\[
S\left( {p,q} \right) = \frac{1}{{q!}}\sum\limits_{l = 0}^q {\left( { - 1} \right)^l \binom{q}{l}\left( {q - l} \right)^p } ,
\]
we immediately obtain our second main result.
\begin{corollary} The Stirling coefficients have an exact representation of the form
\begin{equation}
a_k  = \frac{\left(2k\right)!}{2^k k!}\sum\limits_{i = 0}^{2k} {\binom{k + i - 1/2}{i}\binom{3k + 1/2}{2k - i}2^{i}\sum\limits_{j = 0}^i {\binom{i}{j}\frac{\left( { - 1} \right)^j}{\left( {2k + i + j} \right)!}}\sum\limits_{l = 0}^j {\left( { - 1} \right)^l \binom{j}{l}\left( {j - l} \right)^{2k + i + j} }} .
\end{equation}
\end{corollary}

Note that this formula contains only the four basic operations. To prove Theorem \ref{thm1} we need some concepts. Let $r \geq 0$ and $a_r \neq 0$, let $F\left( x \right) = \sum\nolimits_{j \ge r} {a_j x^j / j!}$ be a formal power series. The potential polynomials $F^{\left(z\right)}_n$ in the variable $z$ are defined by the exponential generating function
\begin{equation}\label{pp1}
\left( {\frac{{a_r x^r / r!}}{{F\left( x \right)}}} \right)^z  = \sum\limits_{n \ge 0} {F_n^{\left( z \right)} \frac{{x^n }}{{n!}}} .
\end{equation}
For $r \geq 1$, the exponential Bell polynomials $B_{n,i}\left(0,\ldots,0,a_r,a_{r+1},\ldots\right)$ in an infinite number of variables $a_r, a_{r+1},\ldots$ can be defined by
\begin{equation}\label{pp2}
\left( {F\left( x \right)} \right)^i  = i!\sum\limits_{n \ge 0} {B_{n,i} \left( {0, \ldots ,0,a_r ,a_{r + 1} , \ldots } \right)\frac{{x^n }}{{n!}}} .
\end{equation}
The following theorem is due to Howard \cite{ref2}.
\begin{theorem} If $F^{\left( z \right)}_n$ is defined by \eqref{pp1} and $B_{n,i}$ is defined by \eqref{pp2}, then
\begin{equation}
F_n^{\left( z \right)}  = \sum\limits_{i = 0}^n {\left( { - 1} \right)^i \binom{z + i - 1}{i}\binom{z + n}{n - i}\left( {\frac{{r!}}{{a_r }}} \right)^{i}\frac{{n!i!}}{{\left( {n + ri} \right)!}}B_{n + ri,i} \left( {0, \ldots ,0,a_r ,a_{r + 1} , \ldots } \right)} .
\end{equation}
\end{theorem}

Now we prove Theorem \ref{thm1}.

\begin{proof}[Proof of Theorem \ref{thm1}]
Brassesco and M\'{e}ndez showed that if
\begin{equation}
G\left( x \right) = 2\frac{{e^x  - x - 1}}{{x^2 }} = 2\sum\limits_{j \ge 0} {\frac{{x^j }}{{\left( {j + 2} \right)!}}},
\end{equation}
then
\begin{equation}\label{pf1}
a_k  = \frac{1}{{2^k k!}}\partial ^{2k} \left( {G^{ - \frac{{2k + 1}}{2}} } \right)\left( 0 \right),
\end{equation}
where $\partial ^{k}f$ denotes the $k$th derivative of a function $f$. Define the polynomials $G^{\left(z\right)}_n$ in the variable $z$ by the following exponential generating function:
\begin{equation}
\left( {\frac{1}{2}\frac{{x^2 }}{{e^x  - x - 1}}} \right)^z  = \sum\limits_{j \ge 0} {G_j^{\left( z \right)} \frac{{x^j }}{{j!}}} .
\end{equation}
Inserting $z=\frac{2k+1}{2}$ into this expression gives
\begin{equation}\label{pf2}
\sum\limits_{j \ge 0} {G_j^{\left( {\frac{{2k + 1}}{2}} \right)} \frac{{x^j }}{{j!}}}  = \left( {\frac{1}{2}\frac{{x^2 }}{{e^x  - x - 1}}} \right)^{\frac{{2k + 1}}{2}}  = \left( {2\frac{{e^x  - x - 1}}{{x^2 }}} \right)^{ - \frac{{2k + 1}}{2}}  = G^{ - \frac{{2k + 1}}{2}} \left( x \right).
\end{equation}
On the other hand we have by series expansion
\begin{equation}\label{pf3}
G^{ - \frac{{2k + 1}}{2}} \left( x \right) = \sum\limits_{j \ge 0} {\partial ^j \left( {G^{ - \frac{{2k + 1}}{2}} } \right)\left( 0 \right)\frac{{x^j}}{{j!}}} .
\end{equation}
Equating the coefficients in \eqref{pf2} and \eqref{pf3} gives
\[
\partial ^j \left( {G^{ - \frac{{2k + 1}}{2}} } \right)\left( 0 \right) = G_j^{\left( {\frac{{2k + 1}}{2}} \right)}  = G_j^{\left( {k + \frac{1}{2}} \right)} .
\]
Now by comparing this with \eqref{pf1} yields
\begin{equation}\label{pf4}
a_k  = \frac{1}{{2^k k!}}G_{2k}^{\left( {k + \frac{1}{2}} \right)} .
\end{equation}
Putting $r=2$ an $a_r = a_{r+1} = \ldots = 1$ into the formal power series $F\left( x \right) = \sum\nolimits_{j \ge r} {a_j x^j / j!}$ gives $F\left( x \right) = e^{x}-x-1$. And therefore the generated potential polynomials are
\[
\left( {\frac{{x^2 / 2!}}{{e^x  - x - 1}}} \right)^z  = \left( {\frac{1}{2}\frac{{x^2 }}{{e^x  - x - 1}}} \right)^z  = \sum\limits_{j \ge 0} {G_j^{\left( z \right)} \frac{{x^j }}{{j!}}} .
\]
According to Howard's theorem we find
\begin{equation}
G_n^{\left( z \right)}  = \sum\limits_{i = 0}^n {\left( { - 1} \right)^i \binom{z + i - 1}{i}\binom{z + n}{n - i}2^{i}\frac{{n!i!}}{{\left( {n + 2i} \right)!}}B_{n + 2i,i} \left( {0, 1 ,1, \ldots } \right)} .
\end{equation}
Now we derive an expression for the exponential Bell polynomials $B_{n,i}\left(0,1,1,\ldots\right)$ in terms of the Stirling numbers of the second kind:
\begin{align*}
i!\sum\limits_{n \ge 0} {B_{n,i} \left( {0, 1, 1 , \ldots } \right)\frac{{x^n }}{{n!}}} & = \left( {F\left( x \right)} \right)^i  = \left( {e^x  - x - 1} \right)^i \\
& = \left( { - x + \sum\limits_{l \ge 1} {\frac{{x^l }}{{l!}}} } \right)^i  = \sum\limits_{j = 0}^i {\binom{i}{j}\left( { - 1} \right)^{i - j} x^{i - j} \left( {\sum\limits_{l \ge 1} {\frac{{x^l }}{{l!}}} } \right)^j } \\
& = \sum\limits_{j = 0}^i {\binom{i}{j}\left( { - 1} \right)^{i - j} x^{i - j} j!\sum\limits_{n \ge 0} {S\left( {n,j} \right)\frac{{x^n }}{{n!}}} } \\
& = \sum\limits_{n \ge 0} {\sum\limits_{j = 0}^i {\binom{i}{j}\left( { - 1} \right)^{i - j} j!S\left( {n,j} \right)\frac{{x^{n + i - j} }}{{n!}}} }\\
& = i!\sum\limits_{n \ge 0} {\left\{ {\frac{{n!}}{{i!}}\sum\limits_{j = 0}^i {\binom{i}{j}\left( { - 1} \right)^{i - j} j!} \frac{{S\left( {n - i + j,j} \right)}}{{\left( {n - i + j} \right)!}}} \right\}\frac{{x^n }}{{n!}}} .
\end{align*}
Hence
\begin{equation}
B_{n,i} \left( {0,1,1 , \ldots } \right) = \frac{{n!}}{{i!}}\sum\limits_{j = 0}^i {\binom{i}{j}\left( { - 1} \right)^{i - j} j!} \frac{{S\left( {n - i + j,j} \right)}}{{\left( {n - i + j} \right)!}}.
\end{equation}
Thus we obtain
\begin{equation}
G_n^{\left( z \right)}  = \sum\limits_{i = 0}^n {\binom{z + i - 1}{i}\binom{z + n}{n - i}2^{i}n!\sum\limits_{j = 0}^i {\binom{i}{j}\left( { - 1} \right)^j j!} \frac{{S\left( {n + i + j,j} \right)}}{{\left( {n + i + j} \right)!}}} .
\end{equation}
Substituting $z=k+1/2$ and $n=2k$ into this expression yields
\begin{equation}
G_{2k}^{\left( k+\frac{1}{2} \right)}  = \sum\limits_{i = 0}^{2k} {\binom{k + i - 1/2}{i}\binom{3k + 1/2}{2k - i}2^{i}\left(2k\right)!\sum\limits_{j = 0}^i {\binom{i}{j}\left( { - 1} \right)^j j!} \frac{{S\left( {2k + i + j,j} \right)}}{{\left( {2k + i + j} \right)!}}},
\end{equation}
hence by (\ref{pf4}) we finally have
\begin{equation}
a_k  = \frac{\left(2k\right)!}{2^k k!}\sum\limits_{i = 0}^{2k} {\binom{k + i - 1/2}{i}\binom{3k + 1/2}{2k - i}2^{i}\sum\limits_{j = 0}^i {\binom{i}{j}\left( { - 1} \right)^j j!} \frac{{S\left( {2k + i + j,j} \right)}}{{\left( {2k + i + j} \right)!}}}.
\end{equation}
This completes the proof of the theorem.
\end{proof}

\section*{Acknowledgment}
I am grateful to Lajos L\'{a}szl\'{o}, who drew my attention to the paper of Brassesco and M\'{e}ndez.

\end{document}